\newcommand{\Holder}{H\"older }

\newcommand{\Emb}{{E}}

\newcommand{\Pmb}{{P}}

\newcommand{\Rmb}{{\mathbb{R}}}

\newcommand{\Cmc}{{\mathcal{C}}}

\newcommand{\Fmc}{{\mathcal{F}}}

\newcommand{\Lmc}{{\mathcal{L}}}

\newcommand {\dagg}{\dagger}

\newcommand {\dens}{\varsigma}
\newcommand {\dist}{\mbox{dist}}

\documentclass[10pt]{article}
\usepackage[portrait,margin=2.54cm]{geometry}

\usepackage{amsfonts}

\usepackage{amssymb}
\usepackage{mathrsfs}
\usepackage{soul}
\usepackage{hyperref}
\usepackage{amssymb,amsthm,amsmath,amsfonts,amsbsy,latexsym}
\usepackage{graphicx}
\usepackage[numeric,initials,nobysame]{amsrefs}
\usepackage{upref,setspace}
\usepackage{enumerate}
\usepackage{enumitem}
\usepackage{color}
\usepackage{enumitem}
\usepackage{mathtools}

\numberwithin{equation}{section}
\numberwithin{figure}{section}
\numberwithin{table}{section}
\newtheorem{lemma}{Lemma}[section]

\newtheorem{theorem}[lemma]{Theorem}
\newtheorem{assumption}[lemma]{Assumption}

\newtheorem{remark}[lemma]{Remark}

\newtheorem{definition}[lemma]{Definition}
\numberwithin{equation}{section}

\mathtoolsset{showonlyrefs}
\begin{document}

\title{Quasistationary Distributions and Ergodic Control Problems}
\author{Amarjit Budhiraja, Paul Dupuis, Pierre Nyquist, Guo-Jhen Wu}
\maketitle

\begin{abstract}
We introduce and study the basic properties of two 
ergodic stochastic control problems associated with the quasistationary distribution (QSD) of a diffusion process $X$ relative to a bounded domain. 
The two problems are in some sense dual, with one
defined in terms of the generator associated with $X$ and the other in terms
of its adjoint. 
Besides proving wellposedness of the associated Hamilton-Jacobi-Bellman equations,
we describe how they can be used to characterize important properties of the QSD.
Of particular note is that the QSD itself can be identified,  up to normalization, in terms of the cost potential of the control problem associated with the adjoint. 

\noindent\newline


\end{abstract}

%

\section{Introduction}

\label{sec:intro}  
Consider a time homogeneous\ Markov process $\{X(t),t\geq
0\}$ with state space $\mathbb{R}^{d}$, and an open subset $G\subset \mathbb{%
R}^{d}$. Suppose the process is started with a distribution $\gamma $, and
that if we condition on the process remaining in $G$ for any interval $[0,t]$%
, $t\in (0,\infty )$, then the conditioned distribution of $X(t)$ is also $%
\gamma $. Such a distribution, if it exists, is called a quasistationary
distribution (QSD) for $X$ and $G$. As with ordinary stationary
distributions, QSDs are of interest in that, under appropriate uniqueness and
communication conditions, they can characterize (conditional) behavior over
long time intervals when starting with more general initial distributions.
However, QSDs are in some ways more complicated than stationary
distributions. For example, assuming that the probability to remain in $G$
decays, then characterizing the QSD will also involve characterizing this
decay rate $\lambda >0$. As another example, it turns out that uniqueness
issues are in general more subtle for QSDs, in that even for very nice sets
and for processes whose communication properties guarantee that there is
only one stationary distribution, it can nonetheless happen that there are
many quasistationary distributions. For general references we refer to the book \cite{colmarmar},
the survey paper \cite{melvil} and the recent work \cite{champagnat2017general}.

The purpose of this paper is to identify and study the basic properties of
two stochastic control problems that characterize important quantities
associated with a QSD when $G$ is a bounded domain. The two problems are in some sense dual, with one
defined in terms of the generator associated with $X$ and the other in terms
of its adjoint. Both problems identify the rate $\lambda $ with an optimal
average cost per unit time (or ergodic cost). 
In both of these control problems the admissible controls are absolutely continuous non-anticipative processes under which the state process remains within $G$ at all times a.s. As a consequence the controls are necessarily unbounded and the associated Hamilton-Jacobi-Bellman equations are given in terms of a boundary condition which requires the solution to diverge to $\infty$ at the boundary of $G$. We prove the wellposedness of these equations and characterize the solutions as the optimal cost and the cost potential (also referred to as the value
function) of the associated ergodic control problems. 

Next, we identify the optimally controlled state processes corresponding to the two control problems.
The optimally controlled dynamics for the problem phrased in terms of the generator of $X$ has appeared in previous works and, under suitable conditions on the model (see \cite{pin} and Remark \ref{Qprocess}),  can be identified as the associated $Q$-process (see \cite{chavil2}) for the underlying diffusion. To the best of our knowledge the optimally controlled diffusion for the ergodic control problem associated with the adjoint of the generator of $X$ has not appeared in previous works on quasistationary distributions and $Q$-processes, although a direct construction of the semigroup of this Markov process can be found in \cite[Section 3]{goqizh}. The two optimally controlled state processes are closely related in that the second process is simply the time reversal of the first when initialized with its unique stationary distribution (see 
Remark \ref{rem:longrem}(3) for a detailed discussion).
For the second ergodic control problem, we
show that the value function for the control problem  identifies the QSD, at least up to normalization.

In forthcoming work we will make two interesting uses of the
representations. One application is to the study of the non-uniqueness issues mentioned
previously, which can occur when the domain $G$ is unbounded. We expect that
different QSDs for the same process and domain can be distinguished
according to different classes of controls and costs that are allowed in the ergodic
control problem, in a way reminiscent of other well known examples of
non-uniqueness in control theory and the calculus of variations \cite{heimiz}.
The other application is to numerical approximation. Using the
Markov chain approximation method for stochastic control problems \cite{kusdup1}, we will construct and prove the
convergence of schemes for approximating the rate $\lambda $ and the QSD $%
\gamma $. These approximations appear to be novel in this context. With the goal of constructing efficient numerical schemes, in the current work we in fact study a collection of ergodic control problems parametrized by a smooth strictly positive function $\dens$ on $\mathbb{R}^d$. These control problems are associated with the adjoints of the generator of the underlying diffusion on $L^2(\mathbb{R}^d, \dens(x) dx)$ for the different choices of $\dens$. It turns out (see Remark \ref{rem:longrem}(2)) that although the optimally controlled state process for all of these control problems is the same irrespective of the choice of $\dens$, the state dynamics and the cost function depend on $\dens$ in an important manner. One expects that the choice of $\dens$ will significantly impact the performance of the numerical schemes based on the corresponding ergodic control problem. These issues will be investigated in a systematic manner in our forthcoming work.

Although in this paper we focus on the case of diffusion process, the
analagous results can be formulated for continuous time jump-diffusions. The
case of a countable state pure jump Markov process is particularly simple,
so long as the state space $\mathcal{X}\subset \mathbb{R}^d$ does not contain accumulation 
points.
The simplicity is due to the fact that in formulating control problems one can simply
restrict to feedback controls that do not allow the process to exit $\mathcal{X} \cap G$.
The corresponding results will be needed for the numerical approximations
mentioned previously, and results analogous to those proved here for diffusions
can be obtained with simpler proofs.  

The paper is organized as follows. In Section \ref{sec:assuandres} we present our main assumptions and results. In particular, Section \ref{sec:ergcontprob} introduces the two ergodic control problems that are the focus of this work. Theorem \ref{thm:ergcontandexitrate} gives our main results for the ergodic control problem associated with the generator of the diffusion while Theorem \ref{thm:ergcontandqsd} gives our results for the ergodic control problem for the adjoint of the generator. In Theorem \ref{thm:llstar} we collect some well known results for Dirichlet eigenvalue problems for elliptic operators and their relationship with quasistationary distributions and extinction rates.
Proofs of our main results (i.e. Theorems \ref{thm:ergcontandexitrate} and \ref{thm:ergcontandqsd}) are provided in Section \ref{sec:secproofs}.

\subsection{Notation}

The following notation will be used. By a domain $G$ in $\mathbb{R}^{d}$ we
will mean an open connected subset of $\mathbb{R}^{d}$. The boundary of such
a domain will be denoted by $\partial G$. Distance between sets $A$ and $B$
in ${\mathbb{R}}^{d}$ is defined as $\inf \{\Vert x-y\Vert :x\in A,y\in B\}$
and is denoted as $\dist(A,B)$.
If $A$ is a singleton $\{x\}$ we write the
distance simply as $\dist(x,B)$. For a domain $G$ in ${\mathbb{R}}^{d}$, $%
\alpha \in (0,1]$ the \Holder space ${\mathcal{C}}^{\alpha }(\bar{G})$
consists of continuous real functions $u$ on $\bar{G}$ which satisfy 
\begin{equation*}
\lbrack u]_{\alpha }\doteq \sup_{x,y\in G,x\neq y}\frac{\left\vert
u(x)-u(y)\right\vert }{\left\Vert x-y\right\Vert ^{\alpha }}<\infty ,
\end{equation*}%
and for $k\in {\mathbb{N}}_{0}$, the space ${\mathcal{C}}^{k,\alpha }(\bar{G}%
)$ consists of real valued functions with continuous partial derivatives up
to $k$-th order, and such that all the $k$-th order partial derivatives are
in ${\mathcal{C}}^{\alpha }(\bar{G})$. We denote by ${\mathcal{C}}%
_{0}^{k,\alpha }(\bar{G})$ the collection of all $u\in {\mathcal{C}}%
^{k,\alpha }(\bar{G})$ with $u(x)=0$ for $x\in \partial G$. For $k \in \mathbb{N}$ and $\alpha \in (0,1]$, we say the $%
\partial G$ is ${\mathcal{C}}^{k,\alpha }$,
if $\partial G$ is a $d-1$ dimensional ${\mathcal{C}}^{k,\alpha }$ manifold, namely,
each point in $\partial G$ has a neighborhood in which $\partial G$ is the graph of a  ${\mathcal{C}}^{k,\alpha }$ function of $d-1$ of the coordinates (cf. \cite[Section 6.2]{giltru}).
For twice differentiable $h:{\mathbb{R}}%
^{d}\rightarrow {\mathbb{R}}$, the gradient $Dh(x)$ is the vector 
$(h_{x_1}(x), \ldots , h_{x_d}(x))'$ and the Hessian $D^{2}h(x)$ is the $d\times d$ matrix with the $(i,j)$-th entry given as
$h_{x_ix_j}(x)$, where
$h_{x_{i}}\doteq \frac{\partial }{\partial
x_{i}}h$ and $h_{x_{i}x_{j}}\doteq \frac{\partial ^{2}}{\partial
x_{i}\partial x_{j}}h$ for $1\leq i,j\leq d$. Denote by ${\mathcal{C}}_{0}^{2}({\mathbb{R}}^{d})$ the space of all real valued functions on $\mathbb{R}^d$ that are twice continuously differentiable with the function and the derivatives vanishing at $\infty$. The trace of a square matrix $A$ will be denoted as  $\mbox{tr}(A)$. For $z\in {\mathbb{C}}
$, $\mbox{Re}(z)$ is the real part of $z$. 
 For a Polish space $S$, denote
by $\mathcal{P}(S)$  the
 space of probability measures  on $S$
equipped with the topology of weak convergence.  Denote by $\mathcal{C}([0,T]:S)$ (resp., $\mathcal{C}([0,\infty):S)$) the space of $S$-valued continuous functions on $[0,T]$ (resp., $[0,\infty)$) equipped with the uniform (resp., locally uniform) topology.

\section{Assumptions and Main Results}

\label{sec:assuandres}

Let $(\Omega ,\mathcal{F},P,\{\mathcal{F}_{t}\}_{t\geq 0})$ be a
filtered probability space equipped with an $m$-dimensional standard $\{%
\mathcal{F}_{t}\}$-Brownian motion $\{B(t)\}_{t\geq 0}$. Consider the $d$%
-dimensional stochastic differential equation (SDE) given as \begin{equation}
dX(t)=b(X(t))dt+\sigma (X(t))dB(t),\;X(0)=X_{0}  \label{eq:sdemodel}
\end{equation}%
where $X_{0}$ is a $\mathcal{F}_{0}$-measurable random variable and $b:%
\mathbb{R}^{d}\rightarrow \mathbb{R}^{d}$ and $\sigma :\mathbb{R}%
^{d}\rightarrow \mathbb{R}^{d\times m}$ are Lipschitz maps, namely for some $%
\kappa _{1}\in (0,\infty )$ 
\begin{equation}
\label{eqn:lip}
\left\Vert b(x)-b(y)\right\Vert +\left\Vert \sigma (x)-\sigma (y)\right\Vert
\leq \kappa _{1}\left\Vert x-y\right\Vert \mbox{ for all }x,y\in {\mathbb{R}}%
^{d}.
\end{equation}%
Under \eqref{eqn:lip} there is a unique pathwise solution of the SDE %
\eqref{eq:sdemodel} for every $\mathcal{F}_{0}$-measurable random variable $%
X_{0}$. When $X_{0}$ has probability distribution $\mu \in \mathcal{P}(\mathbb{R}^{d})$, we will write the probability measure $P$ as $P_{\mu}$.
When $\mu
=\delta _{x}$ for some $x\in \mathbb{R}^{d}$, we will write $P_{\mu
}$ simply as $P_{x}$. 
Let $G$ be a bounded domain in $\mathbb{R}^d$. We are interested in the
study of the QSD of the Markov family $\{X(\cdot),
\{P_{x}\}_{x \in \mathbb{R}^d}\}$ on the domain $G$, which is defined as follows.

\begin{definition}
A probability measure $\gamma \in \mathcal{P}(G) $ is said to be a
QSD of the Markov family $\{X(\cdot),
\{P_{x}\}_{x \in \mathbb{R}^d}\}$ on the domain $G$
if 
\begin{equation}\label{eq:qsddefn}
P_{\gamma} (X(t) \in A \mid \tau_{\partial} >t) = \gamma(A), 
\mbox{
for every } t \in [0, \infty) \mbox{ and } A \in \mathcal{B}(G),
\end{equation}
where $\tau_{\partial} \doteq \inf \{t\ge 0:X(t) \in G^c\}$.
\end{definition}

Frequently we will refer to $\gamma$ simply as a QSD for $X$.

We now introduce additional assumptions on the coefficients and also a
condition on the domain. Let $a=\sigma \sigma ^{T}$, where $\sigma ^{T}$
denotes the transpose of $\sigma$.

\begin{assumption}
\label{assu:coeff} 
The boundary $\partial G$ is ${\mathcal{C}}^{2,\alpha}$ for some $\alpha \in (0,1]$ and
the following hold for some open set $H$ that contains $\bar{G}$.

\begin{enumerate}
\item For some  $\alpha \in (0,1]$, $\sigma _{i,j}\in {\mathcal{C}}^{2,\alpha
}(H)$, and $b_{i}\in {\mathcal{C}}^{1,\alpha }(%
H)$ for all $i=1,\ldots ,d$ and $j=1,\ldots ,m$.

\item There is a $\kappa _{0}\in (0,\infty )$ such that for every $v\in {%
\mathbb{R}}^{d}$ 
\begin{equation*}
\inf_{x\in H}v\cdot a(x)v\geq \kappa _{0}\Vert v\Vert ^{2}.
\end{equation*}
\end{enumerate}
\end{assumption}

Consider the infinitesimal generator ${\mathcal{L}}$ of the Markov process %
\eqref{eq:sdemodel} which, when restricted to ${\mathcal{C}}%
_{0}^{2}({\mathbb{R}}^{d})$, is given as follows. For $f\in {\mathcal{C}}%
_{0}^{2}({\mathbb{R}}^{d})$ 
\begin{equation}
{\mathcal{L}}f(x)\doteq  b(x)\cdot Df(x)+\frac{1}{2}\mbox{tr}\lbrack
a(x)D^{2}f(x)],\;x\in {\mathbb{R}}^{d}.  \label{eq:gener}
\end{equation}%
The formal adjoint ${\mathcal{L}}^{\ast }$ 
on ${\mathcal{C}}_{0}^{2}({\mathbb{R}}^{d})$ of this operator is defined as
follows. For $f\in {\mathcal{C}}_{0}^{2}({\mathbb{R}}^{d})$ 
\begin{equation}
{\mathcal{L}}^{\ast }f(x)\doteq \beta (x)\cdot Df(x)+\frac{1}{2}\mbox{tr}%
\lbrack a(x)D^{2}f(x)]-c(x)f(x),\;x\in {\mathbb{R}}^{d},  \label{eq:generadj}
\end{equation}%
where 
\begin{equation}
\beta _{i}(x)\doteq -b_{i}(x)+\sum_{j=1}^{d}(a_{ij}(x))_{x_{j}}, \;\; 1 \le i \le d,  \label{eq:defbeta}
\end{equation}%
and 
\begin{equation}
c(x)\doteq -\frac{1}{2}\sum_{i,j=1}^{d}(a_{ij}(x))_{x_{i}x_{j}}+%
\sum_{i=1}^{d}(b_{i}(x))_{x_{i}}.  \label{eq:cterm}
\end{equation}%
%
%
%
%
%
%
%
%

It will be useful to consider the adjoint of ${\mathcal{L}}$ with respect to
measures other than the Lebesgue measure. Let $\dens\in {\mathcal{C}}^{2}({%
\mathbb{R}}^{d})\cap {\mathcal{C}}^{2,\alpha }(\bar{G})$ be such that $\dens%
(x)>0$ for all $x\in {\mathbb{R}}^{d}$. Consider the operator ${\mathcal{L}}%
_{\dens}^{\ast }$ on ${\mathcal{C}}%
_{0}^{2}({\mathbb{R}}^{d})$ defined as follows. For $f\in {\mathcal{C}}_{0}^{2}({%
\mathbb{R}}^{d})$, 
\begin{equation}
{\mathcal{L}}_{\dens}^{\ast }f(x)\doteq \frac{1}{\dens(x)}{\mathcal{L}}%
^{\ast }(f\dens)(x),\;x\in {\mathbb{R}}^{d}.  \label{lstardens}
\end{equation}%
Then ${\mathcal{L}}_{\dens}^{\ast }$ is the adjoint to ${\mathcal{L}}$ on the
space $L^{2}(\dens dx)$. Note that when $\dens\equiv 1$, ${\mathcal{L}}_{%
\dens}^{\ast }={\mathcal{L}}^{\ast }$.

The following result on Dirichlet eigenvalue problems for ${\mathcal{L}}$
and ${\mathcal{L}}^*_{\dens}$ will be used in our work. A $\lambda \in {%
\mathbb{C}}$ is a Dirichlet eigenvalue of $-{\mathcal{L}}$ on the domain $G$%
, if for some nonzero $\phi \in {\mathcal{C}}^2(G)\cap \mathcal{C}(\bar G)$ 
\begin{equation}  \label{eq:ev}
\begin{aligned} -{\mathcal{L}} \phi(x) &= \lambda \phi(x), & x \in G,\\
\phi(x) &= 0, & x \in \partial G. \end{aligned}
\end{equation}
The function $\phi$ is called a Dirichlet eigenvector of $-{\mathcal{L}}$
associated with the eigenvalue $\lambda$. Dirichlet eigenvalues and
eigenvectors of $- {\mathcal{L}}^*_{\dens}$ are defined in a similar manner.
Since  we will only be concerned with Dirichlet eigenvectors and eigenvalues for $-{\mathcal{L}}$
and $- {\mathcal{L}}^*_{\dens}$, henceforth we will drop the adjective ``Dirichlet".
For part 6 of the theorem below we will assume the following in addition to Assumption \ref%
{assu:coeff}.
Assumption \ref{ass:extsmoo} is required in the reference \cite{sito2},
and can be replaced by any set of conditions under which the relation \eqref{eq:llstar} holds.

\begin{assumption}
\label{ass:extsmoo} For $H$ as in Assumption \ref{assu:coeff} and some $\alpha \in (0,1]$ the boundary $\partial G$ is ${\mathcal{%
C}}^{4,\alpha}$ and the functions $\sigma_{i,j} \in {\mathcal{C}}%
^{3,\alpha}(H)$ for all $i = 1, \ldots , d$ and $j=1, \ldots , m$.
\end{assumption}
Denote by $\mathcal{C}^*$ the class of all $\dens \in {\mathcal{C}}^2({\mathbb{R}}^d) \cap {%
\mathcal{C}}^{2,\alpha}(\bar G)$ such that $\dens >0$.
Parts 1-5 of the following theorem are immediate consequences of well-known results. Part 6 has been shown in \cite[Proposition 1.10]{pin}, under certain technical conditions on the coefficients (see Hypothesis 2 and 3 therein) which, for example, are satisfied when $a^{-1}b = D U$ for some $U \in \mathcal{C}^2(\bar G)$. We were unable to find a reference
that covers part 6 in the generality considered here. The proof of Theorem \ref{thm:llstar} is provided in Section \ref{sec:proofthmfirst}.
\begin{theorem}
\label{thm:llstar} 
Suppose Assumption %
\ref{assu:coeff} is satisfied. The following hold.

\begin{enumerate}
\item For any $\dens \in \mathcal{C}^*$, $\tilde v$ is an  eigenvector of $-{\mathcal{L}}^*_{\dens}$
	associated with the eigenvalue $z$ if and only if, $v \doteq \dens \tilde v$ is an eigenvector of $-{\mathcal{L}}^*$
associated with the eigenvalue $z$.
\item There exists a $\lambda \in (0,\infty )$ which is an
eigenvalue of  $-{\mathcal{L}}$ and $-{\mathcal{L}}_{\dens}^{\ast }$ for every $\dens \in \mathcal{C}^*$, and
is such that $\mbox{Re}(\tilde{\lambda}%
)\geq \lambda $ for any other  eigenvalue $\tilde{\lambda}$ of $-{%
\mathcal{L}}$ or $-{\mathcal{L}}_{\dens}^{\ast }$, where $\dens \in \mathcal{C}^*$.
\item There is a unique quasistationary distribution $\nu$ for $X$.
\item There are $\psi , \varphi \in {\mathcal{C}}^{2,\alpha}_0(\bar G)$
such that $\psi(x)>0$ and $\varphi(x)>0$ for all $x\in G$, and $\psi$
[resp., $\varphi$] is an eigenvector of $-{\mathcal{L}}$
[resp., $-{\mathcal{L}}^*$] associated with the eigenvalue $\lambda$.
Furthermore, $\int_G
\varphi(x) dx = 1$ and $\int_G \psi(x) \nu(dx) =1$. Any other
eigenvector associated with the eigenvalue $\lambda$ of $-{\mathcal{L}}$ and 
$-{\mathcal{L}}^*$ is a scalar multiple of $\psi$ and $\varphi $ respectively.
For any $\dens \in \mathcal{C}^*$,
$\tilde \varphi \doteq \varphi/\dens$ is an eigenvector of 
 $-{\mathcal{L}}^*_{\dens}$ associated with the eigenvalue $\lambda$.


\item The function $\psi $ characterizes the exit rate from $G$ associated
with the Markov process \eqref{eq:sdemodel} as follows. For all $x\in G$ 
\begin{equation}
\psi (x)=\lim_{t\rightarrow \infty }e^{\lambda t}P_{x}(\tau
_{\partial }>t).
\end{equation}

\item Suppose in addition that Assumption \ref{ass:extsmoo} is satisfied.
Then the unique quasistationary distribution $\nu$ of the Markov
family $\{X(\cdot ),\{P_{x }\}_{x \in \mathbb{R}%
^{d}}\}$ on the domain $G$ can be characterized as $\nu (dx)= \varphi (x)dx$.
\end{enumerate}
\end{theorem}

\subsection{Two Ergodic Control Problems}

\label{sec:ergcontprob} We now introduce two ergodic control problems, one
which is associated with the operator ${\mathcal{L}}$ and the other
associated with ${\mathcal{L}}^*_{\dens}$ where $\dens \in \Cmc^*$.

Consider the $d$-dimensional SDE on $(\Omega ,\mathcal{F},P,\{%
\mathcal{F}_{t}\}_{t\geq 0})$ 
\begin{equation}
dY(t)=b(Y(t))dt+\sigma (Y(t))dB(t)+\sigma (Y(t))u(t)dt,\;Y(0)=y\in G,
\label{contback}
\end{equation}%
where $B(t)$ is a $m$-dimensional $\Fmc_t$-Brownian motion,
$\{u(t),0\leq t<\infty \}$ is a ${\mathbb{R}}^{m}$-valued $\mathcal{F}%
_{t}$-progressively measurable process satisfying $\int_{0}^{T}\Vert
u(t)\Vert ^{2}dt<\infty $ for every $T<\infty $. The process $u$ is regarded
as a control process modifying the original dynamics in \eqref{eq:sdemodel}.
Such a control process is called an \emph{admissible control} (for initial
condition $y\in G$) if the following conditions are satisfied.

\begin{enumerate}[label={C.\arabic*}]

\item There is a unique continuous, ${\mathbb{R}}^d$-valued, $\mathcal{F}_t$%
-adapted process $Y$ that solves \eqref{contback}.

\item ${P}(Y(t) \in G \mbox{ for all } t \in [0,\infty)) = 1$.
\end{enumerate}

We denote the collection of all admissible controls by ${\mathcal{A}}(y)$.
Although previously the notation $P_y$ and $E_y$ was used for processes $X$ that start from $y$, with an abuse of notation we will continue to use this notation for probability measures (and expectations) under which $Y(0)=y$.
The cost associated with the control problem is defined as follows. For $%
y\in G$ and $u\in {\mathcal{A}}(y)$ 
\begin{equation}
J(y,u)\doteq \limsup_{T\rightarrow \infty }\frac{1}{T}{E}%
_{y}\int_{0}^{T}\left[ \frac{1}{2}\Vert u(t)\Vert ^{2}\right] dt.
\end{equation}%
The optimal cost, which could in principle depend on the initial
condition, is given as follows. For $y\in G$ 
\begin{equation}
\label{eqn:defofV}
J(y)\doteq \inf_{u\in {\mathcal{A}}(y)}J(y,u)=\inf_{u\in {\mathcal{A}}%
(y)}\limsup_{T\rightarrow \infty }\frac{1}{T}{E}_{y}\int_{0}^{T}%
\left[ \frac{1}{2}\Vert u(t)\Vert ^{2}\right] dt.
\end{equation}%
It will be shown in Theorem \ref{thm:ergcontandexitrate} (5) that in fact $%
J(y)$ is independent of $y\in G$. This is a well known fact for many types of ergodic control problems and the proof here is a variation on classical arguments that accounts for the requirement C.2 of admissible controls and the unboundedness of the control space.
The theorem will also give a
characterization of $J$ in terms of the following Hamilton-Jacobi-Bellman
(HJB) equation:%
\begin{equation}
\begin{aligned} \frac{1}{2}\mbox{tr}[a(x)D^2 \Gamma(x)] + \min_{u\in
\Rmb^m}\left[ (b(x) + \sigma(x) u)\cdot D \Gamma(x) + \frac{1}{2} \|u\|^2 \right]
&= \gamma, \;\; x \in G,\\ \Gamma(x) \to \infty &\mbox{ as } x \to \partial G.
\end{aligned}  \label{eq:hjbfirst}
\end{equation}%
By a solution $(\Gamma ,\gamma )$ of \eqref{eq:hjbfirst} we mean a $\Gamma \in {%
\mathcal{C}}^{2}(G)$ and $\gamma \in (0,\infty )$ that satisfy the first
line of the equation for each $x\in G$ and such that the divergence
statement in the second line holds.
In a more classical setup where either $G= \mathbb{R}^d$ or perhaps a reflecting boundary condition is used, $\gamma$ identifies the optimal ergodic cost regardless of the starting state, while $\Gamma$, which is typically unique only up to an additive constant,
gives what is called the {\it cost potential}. Under suitable conditions, the cost potential generates through the equation an optimal feedback control. 
The cost potential also indicates the impact of the initial condition over finite but large time intervals, 
in that heuristically one expects (and can under conditions show rigorously) that if $J_T(x)$ is the minimal cost over $[0,T]$ when starting at $x$ at time $t=0$,
then $J_T(x)-J_T(y) \rightarrow \Gamma(x) -\Gamma(y)$ as $T\rightarrow \infty$.

The following theorem relates the ergodic control problem \eqref{eqn:defofV}
with the
exit rate of the Markov process $X$ from the domain $G$. 

%
%

\begin{theorem}
\label{thm:ergcontandexitrate} Suppose Assumption \ref{assu:coeff} is
satisfied. Then the following hold.

\begin{enumerate}
\item Let $\psi$ and $\lambda$ be as given by Theorem \ref{thm:llstar}.
Define $\Psi \doteq -\log \psi$. Then the pair $(\Psi,\lambda)$ solves the
HJB equation \eqref{eq:hjbfirst}.

\item Suppose $(\tilde\Psi, \tilde\lambda) \in {\mathcal{C}}^2(G)\times
(0,\infty) $ is a solution of the HJB equation \eqref{eq:hjbfirst}. Then $%
\tilde\lambda = \lambda$ and for some $C \in {\mathbb{R}}$, $%
\Psi(x)-\tilde\Psi(x) = C$ for all $x \in G$.

\item Consider the control $\{u(t),0\leq t<\infty \}$ defined in feedback
form by
\begin{equation}
u(t)=-\sigma (Y(t))^{T}D\Psi (Y(t)),\;0\leq t<\infty ,  \label{eq:optcont}
\end{equation}%
where $Y$ is defined by \eqref{contback} with this choice of $u$,
namely 
\begin{equation}
dY(t)=b(Y(t))dt+\sigma (Y(t))dB(t)-a(Y(t))D\Psi (Y(t))dt,\;Y(0)=y\in G.
\label{contback2}
\end{equation}%
Then there is a unique solution to \eqref{contback2} which satisfies ${%
P}(Y(t)\in G\mbox{ for all }t\in \lbrack 0,\infty ))=1$. In
particular, $u\in {\mathcal{A}}(y)$.

\item The process $Y$
defines a Markov process on $G$ that 
has a unique invariant probability measure $\mu $ given by $\mu (dx)=\varphi
(x)\psi (x)dx$, where $\varphi$ is from Theorem \ref{thm:llstar}.

\item  $J(y)=\lambda $
for all $y\in G$.

\item  The admissible
control process defined by \eqref{eq:optcont} is optimal, and in fact 
\begin{equation}
\begin{aligned}
\lambda &=\lim_{T\rightarrow \infty }\frac{1}{T}{E}_{y}\int_{0}^{T}%
\frac{1}{2}\left[ D\Psi (Y(t))\cdot a(Y(t))D\Psi (Y(t))\right] dt.
\end{aligned}
\end{equation}%
Furthermore, the eigenvalue $\lambda $ can be determined from the invariant
measure $\mu $ by the formula 
\begin{equation}
\lambda =\frac{1}{2}\int_{G}[D\Psi (y)\cdot a(y)D\Psi (y)]\mu (dy).
\label{eq:lamdainv}
\end{equation}
\end{enumerate}
\end{theorem}
\begin{remark}
	\label{Qprocess}
	For fixed $0\le t \le T$, let $\{Y^T(s)\}_{0\le s \le t}$ be a $G$-valued stochastic process  with probability law $Q^{T,t} \in \mathcal{P}(\mathcal{C}([0,t]:G))$  given as 
	$$Q^{T,t}(A) \doteq P_y(X(\cdot) \in A \mid \tau_{\partial} >T), \;\; A \in \mathcal{B}(\mathcal{C}([0,t]:G)).$$
	In \cite{pin} it is shown that under certain conditions on the coefficients (see Hypothesis 1 and 2 therein)
	 which, for example, are satisfied when $a^{-1}b = D U$ for some $U \in \mathcal{C}^1(\bar G)$, as 
	 $T\to \infty$, $Q^{T,t}$ converges in $\mathcal{P}(\mathcal{C}([0,t]:G))$ to some $Q^t$, for every 
	 $t\in 
	 (0,\infty)$. This yields
	 a probability measure $Q$ on  $\mathcal{C}([0,\infty):G)$ such that the corresponding induced measure on $\mathcal{C}([0,t]:G)$ is $Q^t$ for every $t$.
	 Later works \cite{goqizh, chavil2} have proved similar results under more general conditions.
	 The corresponding process whose probability law is $Q$ is sometimes referred to as the $Q$-process (cf. \cite{chavil2}). The paper \cite{pin}  shows that the $Q$-process has the same law as the diffusion $Y$ defined by
	 \eqref{contback2} (see also \cite{goqizh} where the measure $Q$ is characterized through the semigroup of the associated Markov process). Thus the process given in \eqref{contback2} reduces, under  conditions of \cite{pin}, to the $Q$-process associated with the diffusion \eqref{eq:sdemodel} and the domain $G$.
\end{remark}
We next introduce an ergodic control problem that can be used to determine
the quasistationary distribution of $X$.

Recall the operator ${\mathcal{L}}_{\dens}^{\ast }$ from \eqref{lstardens}.
Then, for $f\in {\mathcal{C}}_{0}^{2}({\mathbb{R}}^{d})$
\begin{equation}
{\mathcal{L}}_{\dens}^{\ast }f(x) =  \tilde{\beta}(x)\cdot Df(x)+\frac{1}{%
2}\mbox{tr}\lbrack a(x)D^{2}f(x)]-\tilde{c}(x)f(x),\;x\in {\mathbb{R}}^{d},
\label{eq:generadjden}
\end{equation}%
where 
\begin{equation}\label{eq:tilbeta}
\tilde{\beta}(x)\doteq \beta (x)+\frac{1}{\dens(x)}a(x)D\dens(x)
\end{equation}%
and 
\begin{equation}
\tilde{c}(x) \doteq c(x)-\frac{1}{\dens(x)}\beta (x)\cdot D\dens(x)-\frac{1}{2\dens%
(x)}\mbox{tr}\lbrack a(x)D^{2}\dens(x)] . \label{eq:ctermn}
\end{equation}

Consider the $d$-dimensional SDE on $(\Omega ,\mathcal{F},P,\{%
\mathcal{F}_{t}\}_{t\geq 0})$ given by 
\begin{equation}
dZ(t)=\tilde{\beta}(Z(t))dt+\sigma (Z(t))dB(t)+\sigma
(Z(t))u(t)dt,\;Z(0)=z\in G.  \label{contbackstarden}
\end{equation}%
where $\{u(t),0\leq t<\infty \}$ is a ${\mathbb{R}}^{m}$-valued $\mathcal{F}%
_{t}$-progressively measurable process satisfying $\int_{0}^{T}\Vert
u(t)\Vert ^{2}dt<\infty $ for every $T<\infty $. Such a control process is
called an \emph{admissible control} (for initial condition $z\in G$) if the
following conditions are satisfied.

\begin{enumerate}[label={$\mbox{C}^{\prime }$.\arabic*}]

\item There is a unique continuous, ${\mathbb{R}}^d$-valued, $\mathcal{F}_t$%
-adapted process $Z$ that solves \eqref{contbackstarden}.

\item ${P}(Z(t) \in G \mbox{ for all } t \in [0,\infty)) = 1$.
\end{enumerate}

We denote the collection of all admissible controls as ${\mathcal{A}}%
^{* }(z)$. Once again we will write $P_z$ and $E_z$ for probabilities and expectations under which $Z(0)=z$ a.s.
The cost associated with the control problem is defined as
follows. For $z\in G$ and $u\in {\mathcal{A}}^{* }(z)$ 
\begin{equation}
\label{jstarzu}
J^{* }(z,u)\doteq \limsup_{T\rightarrow \infty }\frac{1}{T}{E}%
_{z}\int_{0}^{T}\left[ \frac{1}{2}\Vert u(t)\Vert ^{2}+\tilde{c}(Z(t))\right]
dt,
\end{equation}%
where $\tilde{c}$ is defined in \eqref{eq:ctermn}. The optimal cost $%
J^{* }$ of the control problem is given as follows. For $z\in G$ 
\begin{equation}
\label{eqn:erg_contr_adjoint}
J^{* }(z)\doteq \inf_{u\in {\mathcal{A}^*}(z)}J^{* }(z,u)=\inf_{u\in 
{\mathcal{A}^*}(z)}\limsup_{T\rightarrow \infty }\frac{1}{T}{E}%
_{z}\int_{0}^{T}\left[ \frac{1}{2}\Vert u(t)\Vert ^{2}+\tilde{c}(Z(t))\right]
dt.
\end{equation}%
Similar to the case of $J(y)$ in \eqref{eqn:defofV} and part 5 of Theorem \ref{thm:ergcontandexitrate}, in Theorem \ref{thm:ergcontandqsd} (5) it will be shown that the quantity $%
J^{* }(z)$ is independent of $z\in G$. Consider now the HJB equation
associated with the control problem \eqref{eqn:erg_contr_adjoint}:
\begin{equation}
\begin{aligned} \frac{1}{2}\mbox{tr}[a(x)D^2 \Gamma(x)] + \min_{u\in
\Rmb^d}\left[ (\tilde\beta(x) + \sigma(x) u)\cdot D \Gamma(x) + \tilde
c(x)+\frac{1}{2} \|u\|^2 \right] &= \gamma, \;\; x \in G,\\ \Gamma(x) \to
\infty &\mbox{ as } x \to \partial G. \end{aligned}  \label{eq:hjbfirststar}
\end{equation}%
As before, by a solution $(\Gamma ,\gamma )$ of \eqref{eq:hjbfirststar} we mean a $%
\Gamma \in {\mathcal{C}}^{2}(G)$ and $\gamma \in (0,\infty )$ that satisfy the
first line of the equation for each $x\in G$ and such that the divergence
statement in the second line holds.

Together with Theorem \ref{thm:llstar}, the following theorem relates this ergodic control problem with the QSD
of the Markov process $X$.

\begin{theorem}
\label{thm:ergcontandqsd} Suppose that Assumption \ref{assu:coeff}
is satisfied. Let $\dens \in \Cmc^*$.
Then the following hold.

\begin{enumerate}
\item Let $\tilde\varphi$ and $\lambda$ be as given by Theorem \ref%
{thm:llstar}. Define $\tilde\Phi \doteq -\log \tilde\varphi$. Then the pair $%
(\tilde\Phi,\lambda)$ solve the HJB equation  \eqref{eq:hjbfirststar}.

\item Suppose $(\bar\Phi, \bar\lambda) \in {\mathcal{C}}^2(G)\times (0,\infty)$
is a solution of the HJB equation  \eqref{eq:hjbfirst}. Then $\bar\lambda
= \lambda$ and for some $C \in {\mathbb{R}}$, $\tilde\Phi(x)-\bar\Phi(x) = C$
for all $x \in G$.

\item Consider the control $\{u(t),0\leq t<\infty \}$ defined in feedback
form by
\begin{equation}
u(t)=-\sigma (Z(t))^{T}D\tilde{\Phi}(Z(t)),\;0\leq t<\infty
\label{eq:optcontstar}
\end{equation}%
where $Z$ is defined by \eqref{contbackstarden} with this choice of $u$%
, namely 
\begin{equation}
dZ(t)=\tilde{\beta}(Z(t))dt+\sigma (Z(t))dB(t)-a(Z(t))D\tilde{\Phi}%
(Z(t))dt,\;Z(0)=z\in G.  \label{contbackstar2}
\end{equation}%
Then there is a unique solution to \eqref{contbackstar2} which satisfies 
$P_z(Z(t)\in G\mbox{ for all }t\in \lbrack 0,\infty ))=1$. In
particular, $u\in {\mathcal{A}}^{* }(z)$.

\item The process $Z$ in \eqref{contbackstar2} defines a Markov process on $G$ that 
has a unique invariant probability measure  $\mu $ given by $\mu (dx)=\varphi
(x)\psi (x)dx$, where $\varphi$ and $\psi$ are from Theorem \ref{thm:llstar}.

\item The optimal cost
satisfies $J^{* }(z)=\lambda $ for all $z\in G$.

\item  The admissible
control process defined by \eqref{eq:optcontstar} is optimal, and in fact 
\begin{align*}
\lambda 
& =\lim_{T\rightarrow \infty }\frac{1}{T}{E}\int_{0}^{T}\left[ 
\frac{1}{2}D\tilde\Phi (Z(t))\cdot a(Z(t))D\tilde\Phi (Z(t))+\tilde{c}(Z(t))\right] dt.
\end{align*}%
Furthermore, the eigenvalue $\lambda $ can be determined from the invariant
measure $\mu$ by the formula 
\begin{equation}
\lambda =\frac{1}{2}\int_{G}\left[ D\tilde{\Phi} (y)\cdot a(y)D\tilde{\Phi} (y)+\tilde{c}%
(y)\right] \mu (dy).  \label{eq:lamdainv2}
\end{equation}
\end{enumerate}
\end{theorem}

\begin{remark}$\;$
\label{rem:longrem}
\begin{enumerate}
	\item Although the $G$-valued diffusion $Y$ in \eqref{contback2} has been previously studied in relation to $Q$-processes (e.g. in \cite{pin}), to the best of our knowledge, the diffusion identified in \eqref{contbackstar2} has not appeared in previous literature on quasistationary distributions.
	\item Note that for every $z\in G$, with $\Phi = -\log \varphi$,
$$\tilde \beta(z) -a(z)D\tilde{\Phi}(z) = \beta(z) - a(z)D{\Phi}(z).$$
Thus the equation for the process $Z$ in \eqref{contbackstar2} can be written as
\begin{equation}\label{eq:newzeq}
dZ(t)={\beta}(Z(t))dt+\sigma (Z(t))dB(t)-a(Z(t))D{\Phi}(Z(t))dt,\;Z(0)=z\in G.  
\end{equation}%
It follows that the optimally controlled process $Z$ is the same for all $\dens \in \Cmc^*$. However note that the control problem formulated in \eqref{jstarzu}- \eqref{eqn:erg_contr_adjoint} is different for different choices of $\dens$. To see this we rewrite the controlled evolution as 
\begin{align*}
dZ(t)&=\tilde{\beta}(Z(t))dt+\sigma (Z(t))dB(t)+\sigma
(Z(t))u(t)dt\\
&={\beta}(Z(t))dt+\sigma (Z(t))dB(t)+\sigma(Z(t))\left[u(t) + \frac{1}{\dens(Z(t))}\sigma^T(Z(t))D\dens(Z(t))\right]dt\\
&= {\beta}(Z(t))dt+\sigma (Z(t))dB(t)+\sigma(Z(t))\tilde u(t) dt
\end{align*}%
where
$$\tilde u(t) \doteq u(t) + \frac{1}{\dens(Z(t))}\sigma^T(Z(t))D\dens(Z(t)).$$
The expectation of the integrand in the cost function can be rewritten as
\begin{align*}
	E_{z}\left[ \frac{1}{2}\Vert u(t)\Vert ^{2}+\tilde{c}(Z(t))\right]
	&={E}_{z}\left[ \frac{1}{2}\left\Vert \tilde u(t)- \frac{1}{\dens(Z(t))}\sigma^T(Z(t))D\dens(Z(t))\right\Vert^{2}\right. \\
	&\quad +\left. 
	c(Z(t)) -\frac{1}{\dens(Z(t))}\beta (Z(t))\cdot D\dens(Z(t))-
	\frac{1}{2\dens(Z(t))}\rm{tr}\lbrack a(Z(t))D^{2}\dens(Z(t))]\right].
\end{align*}
With this rewriting, the controlled state dynamics is the same across all $\dens \in \Cmc^*$, however as seen from the above display, the cost function depends on the first two derivatives of $\dens$.
\item 
We note the somewhat surprising fact, previously remarked on in \cite{goqizh}, that
according to part 4 of Therorem \ref{thm:ergcontandqsd} and part 4 of Theorem \ref{thm:ergcontandexitrate},
the stationary distributions of the two optimally controlled processes,
one associated with ${\mathcal{L}}$ and the other
with ${\mathcal{L}}^*_{\dens}$,
are in fact the same.
Indeed, there is a close connection between the diffusions in \eqref{contback2}
and \eqref{contbackstar2}.
Specifically, denoting the infinitesimal generators of the diffusions in \eqref{contback2} and \eqref{contbackstar2}
by $\Lmc_Y$ and $\Lmc_Z$, it can be checked that, for 
$f\in {\mathcal{C}}_{0}^{2}(G)$,
\begin{equation}\Lmc_Z(f) = \frac{1}{\eta}\Lmc_Y^*(\eta f),\label{eq:llstarn}\end{equation}
where $\eta = \varphi \psi$ is the invariant density of the two diffusions. To see this note that, for $f\in {\mathcal{C}}_{0}^{2}(G)$,
\begin{align*}
\Lmc_Y^*(f) &= 
\hat{\beta}\cdot Df + \frac{1}{2}{\rm tr}\lbrack aD^{2}f] -\hat{c}f
\end{align*}
where 
$$\hat{\beta} = \beta + a D\Psi, \;\; 
\hat{c} = c - \sum_{i=1}^d [(aD\Psi)_i]_{x_i}.
$$
Thus
\begin{align*}
\frac{1}{\eta}\Lmc_Y^*(\eta f) &= 
\frac{1}{\eta}\left[\hat{\beta}\cdot D(\eta f) + 
\frac{1}{2}{\rm tr}\lbrack aD^{2}(\eta f)] - \hat{c}\eta f
\right]\\
&= \hat{\beta} Df + \frac{1}{2}{\rm tr}\lbrack aD^{2}f]
+ \frac{1}{\eta} D\eta \cdot a Df + \frac{1}{\eta}f\Lmc_Y^*(\eta)\\
&= \beta Df + \frac{1}{2}{\rm tr}\lbrack aD^{2}f]
- Df\cdot a D\Phi\\
&= \Lmc_Z(f),
\end{align*}
where the third equality follows on recalling that $\eta$ is the invariant density for $Y$, and the definitions of $\beta^*$
and $\eta$.

The equality in \eqref{eq:llstarn} suggests that the stationary optimal $Y$ is simply the time-reversal of the stationary optimal $Z$. More precisely, consider the processes
$Y$ (resp. $Z$) solving \eqref{contback2} (resp. \eqref{contbackstar2}) with $Y(0)$ distributed as $\mu$ (resp. $Z(0)$ distributed as $\mu$).
Fix $T<\infty$ and for $0\le t \le T$, define $\tilde Z(t) \doteq Y(T-t)$. Then the distributions of $Z$ and $\tilde Z$ on $\mathcal{C}([0,T]:G)$ are the same.
We give the following formal argument, which we believe can be rigorized. It suffices to check that the generator of $Z$ and $\tilde Z$ restricted to $\mathcal{C}_0^2(\bar G)$ are the same.
Denote the transition probability semigroups of $Y$ and $\tilde Z$ by $\{T_t\}$ and $\{\tilde T_t\}$ respectively. Then for $0\le s \le t+s \le T$ and $f,g \in \mathcal{C}_0^2(\bar G)$
$$
E(g(Y(s))f(Y(t+s))) = \int_{G} g(y) T_tf(y)\eta(y) dy = \int_{G} f(y)\tilde T_t g(y) \eta(y) dy.$$
Substracting $\int_{G} f(y) g(y) \eta(y) dy$ from both sides and dividing by $t$ we get
$$
\int_{G} g(y) \frac{[T_tf(y) - f(y)]}{t}\eta(y) dy = \int_{G} f(y)\frac{[\tilde T_t g(y) - g(y)]}{t} \eta(y) dy.$$
Sending $t\to 0$
$$\int_{G} g(y) \mathcal{L}_Yf(y )\eta(y) dy = \int_{G} f(y)\mathcal{L}_{\tilde Z}g(y) \eta(y) dy$$
where $\mathcal{L}_{\tilde Z}$ is the generator of $\tilde Z$.
Noting from \eqref{eq:llstarn} that
$$\int_{G} g(y) \mathcal{L}_Yf(y )\eta(y) dy = \int_{G} \frac{1}{\eta(y)}\mathcal{L}_Y^* [\eta g](y) f(y)\eta(y) dy  = \int_{G} \mathcal{L}_Z g(y) f(y)\eta(y) dy$$
we now see that
$$\int_{G} \mathcal{L}_Z g(y) f(y)\eta(y) dy = \int_{G} f(y)\mathcal{L}_{\tilde Z}g(y) \eta(y) dy$$
for all $f,g \in \mathcal{C}_0^2(\bar G)$ which shows that the generator of $Z$ and $\tilde Z$ restricted to $\mathcal{C}_0^2(\bar G)$ are the same.

In particular when $\Lmc_Y$ is selfadjoint on $L^2(\mu)$, it turns out that the processes $Y$ in \eqref{contback2} and $Z$ in \eqref{contbackstar2} are the same. As a simple example consider the one dimensional diffusion on the bounded interval $[0,K]$ with absorption at $0$ and $K$, given as
 $$X(t) = x+B(t)-at, \; t \in [0, \infty).$$
 In this example, it is easy to check that 
 $\Lmc_Y$ is selfadjoint on $L^2(\mu)$ and both of the optimally controlled processes $Y$ and $Z$ are governed by the stochastic differential equation
 $$dU(t) = \frac{\pi}{K}\cot\left(\frac{\pi}{K} U(t)\right) dt + dB(t).$$
\end{enumerate}
\end{remark}
\section{Proofs}
\label{sec:secproofs}

\subsection{Proof of Theorem \protect\ref{thm:llstar}}
\label{sec:proofthmfirst}
Part 1 is immediate from \eqref{lstardens}.
From the Krein-Rutman
theorem (cf. \cite[Theorems 1.1--1.4]{du2006order}) it follows that
there is a $\lambda \in (0,\infty )$ which is a simple  eigenvalue
of both $-{\mathcal{L}}$ and $-{\mathcal{L}}^{\ast }$ and any other
 eigenvalue $\tilde{\lambda}$ of these operators satisfies $%
\mbox{Re}(\tilde{\lambda})\geq \lambda $. This together with part 1 proves part 2. 
For part 3 see  \cite[Theorem 1.1]{champagnat2017general}.
Next denote by 
$\psi $ and $\varphi $ the unique eigenvectors of $-{\mathcal{L}}$ and $-{%
\mathcal{L}}^{\ast }$, respectively, associated with $\lambda$. Then again from the Krein-Rutman
theorem (cf. \cite[Theorem 1.3]{du2006order}) these functions are strictly
positive and (cf. \cite[Theorem 6.14]{giltru}) are in ${\mathcal{C}}_{0}^{2,\alpha }(\bar{G})$. We normalize
these eigenvectors such that 
\begin{equation}
\int_{G}\varphi (x)dx=1\mbox{ and }\int_{G}\psi (x) \nu(dx)=1.
\end{equation}%
This together with part 1 completes the proof of part 4.

Part 5 follows from 
\cite[Theorem 1.1]{champagnat2017general} (see also \cite[Theorem 3]{goqizh}) on recalling the normalization used to define $\psi$.

We now prove part 6. For this it suffices to show that for all $f\in
\mathcal{C}_{0}^{2}(G)$ and $t>0$, with $\tilde \nu(dx) = \varphi(x) dx$.
\begin{equation}
\frac{E_{\tilde \nu }[f(X(t))1_{\tau _{\partial }>t}]}{E_{\tilde \nu }[1_{\tau _{\partial
}>t}]}=\langle f,\tilde \nu \rangle ,  \label{eq: toshow}
\end{equation}%
where $\langle f,\tilde \nu \rangle \doteq \int_G f(x)\tilde{\nu}(dx).$

To see that it suffices to work with $f\in
\mathcal{C}_{0}^{2}(G)$, note that 
\begin{equation*}
    P_{\tilde \nu }(X(t) \in \partial G, \tau_{\partial }>t)=0 \mbox{ and }\tilde \nu(\partial G)=0.
\end{equation*}
Using the dominated convergence theorem, given any $f$ that is bounded and measurable we can approximate by functions with the same bound which are nonzero only in compact subsets of $G$.
These in turn can be approximated by elements of $\mathcal{C}_{0}^{2}(G)$ using standard methods.

We now show \eqref{eq: toshow} for $f\in \mathcal{C}_{0}^{2}(G)$.
Let $X^{\dagg}(t)$ be the process defined as $X^{\dagg}(t)\doteq X(t\wedge
\tau _{\partial })$ and let $p_{t}^{\dagg}(x,y)$ denote the transition
subprobability density of the process restricted to $G$, namely for $f\in {%
\mathcal{C}}_{0}^{2}(G)$, 
\begin{equation*}
E_{x}(f(X^{\dagg}(t))=\int_{G}f(y)p_{t}^{\dagg}(x,y)dy.
\end{equation*}%
For a function $h(x,y)$ in ${\mathcal{C}}^{2}(G\times G)$, denote by ${%
\mathcal{L}}_{x}h$ the function obtained by regarding ${\mathcal{L}}$ as a
differential operator in the variable $x$. Define ${\mathcal{L}}_{y}h$ in a
similar fashion. Then for all $x,y\in G$ and for all $t>0$ (see \cite[%
Theorem 1]{sito2}) 
\begin{equation}
{\mathcal{L}}_{x}p_{t}^{\dagg}(x,y)={\mathcal{L}}_{y}^{\ast }p_{t}^{\dagg%
}(x,y).  \label{eq:llstar}
\end{equation}%
Next note that for $f\in \mathcal{C}_{0}^{2}(G)$ 
\begin{align*}E_{\tilde\nu }[f(X^{\dagg%
}(t))] &= E_{\tilde\nu }[f(X(t\wedge \tau _{\partial }))]
= \langle f,\tilde\nu \rangle + 
E_{\tilde\nu }\left[ \int_0^{t\wedge \tau _{\partial }}
\mathcal{L}f(X(s)) ds \right]\\
&= \langle f,\tilde\nu \rangle + 
 \int_0^{t }
E_{\tilde\nu }\left[\mathcal{L}f(X^{\dagg}(s))  \right]ds.
\end{align*}
Thus
\begin{align*}
E_{\tilde\nu }[f(X(t))1_{\tau _{\partial }>t}]& =E_{\tilde\nu }[f(X^{\dagg%
}(t))]\\
& =\langle f,\tilde\nu \rangle+\int_{0}^{t}\int_{G}\int_{G}{\mathcal{L}}%
f(y)p_{s}^{\dagg}(x,y)dy\varphi (x)dxds \\
& =\langle f,\tilde\nu \rangle +\int_{0}^{t}\int_{G}\int_{G}f(y){\mathcal{L}}%
_{y}^{\ast }p_{s}^{\dagg}(x,y)dy\varphi (x)dxds.
\end{align*}
Using \eqref{eq:llstar} we have
\begin{align*}
E_{\tilde\nu }[f(X(t))1_{\tau _{\partial }>t}]
& =\langle f,\tilde\nu \rangle +\int_{0}^{t}\int_{G}\int_{G}f(y){\mathcal{L}}%
_{x}p_{s}^{\dagg}(x,y)dy\varphi (x)dxds \\
& =\langle f,\tilde\nu \rangle +\int_{0}^{t}\int_{G}f(y)\int_{G}p_{s}^{\dagg}(x,y){%
\mathcal{L}}^{\ast }\varphi (x)dxdyds \\
& =\langle f,\tilde\nu \rangle -\lambda \int_{0}^{t}\int_{G}f(y)\int_{G}p_{s}^{%
\dagg}(x,y)\varphi (x)dxdyds \\
& =\langle f,\tilde\nu \rangle -\lambda \int_{0}^{t}E_{\tilde\nu }[f(X(s))1_{\tau
_{\partial }>s}]ds.
\end{align*}%
This shows that 
\begin{equation}
E_{\tilde\nu }[f(X(t))1_{\tau _{\partial }>t}]=e^{-\lambda t}\langle f,\tilde\nu \rangle .
\label{eq:odesol}
\end{equation}%
Now let $g_{n}\in \mathcal{C}_{0}^{2}(G)$ be such that $\sup_{n,x}|g_{n}(x)|<\infty $
and $g_{n}(x)\rightarrow 1$ for all $x\in G$. Then from \eqref{eq:odesol},
for all $n$ 
\begin{equation}
E_{\tilde\nu }[g_{n}(X(t))1_{\tau _{\partial }>t}]
=e^{-\lambda t}\langle g_{n},\tilde\nu \rangle.  \label{eq:odesolgn}
\end{equation}%
Sending $n\rightarrow \infty $, by the dominated convergence theorem
\begin{equation}
\Pmb_{\tilde\nu }(\tau _{\partial }>t)=\Emb_{\tilde\nu }[1_{\tau _{\partial
}>t}]=\lim_{n\rightarrow \infty }\Emb_{\tilde\nu }[g_{n}(X(t))1_{\tau _{\partial }>t}]=e^{-\lambda
t}\lim_{n\rightarrow \infty }\langle g_{n},\tilde\nu \rangle =e^{-\lambda t}.
\label{eq:statesit}
\end{equation}
Substituting this in \eqref{eq:odesol} we obtain \eqref{eq: toshow}, which
completes the proof of 6.

 \hfill \qed
\subsection{A technical lemma}
The following lemma will be used to show that  certain stochastic differential equations with drift functions given in terms of a solution of 
\eqref{eq:hjbfirst} or a solution of \eqref{eq:hjbfirststar}  have unique pathwise solutions that remain in the domain $G$ at all times.
\begin{lemma}
\label{lem:lemstayin} 
Under Assumption \ref{assu:coeff} the following hold.
\begin{enumerate}[label={(\alph*)}]
\item Let $\{Y(t)\}$ solve the  SDE 
\begin{equation*}
dY(t)=b(Y(t))dt+\sigma (Y(t))dB(t)-a (Y(t))D\Gamma (Y(t))dt, 
\end{equation*}%
where $(\Gamma ,\gamma )$ is a solution of \eqref{eq:hjbfirst}. 
For $\delta\in (0,\infty)$ let
\begin{equation*}
\tau_{\delta} \doteq \inf \{t>0: \mbox{\rm dist}(Y(t), \partial G) \le \delta\}.
\end{equation*}%
Then for $y \in G$ and any $T\in (0,\infty )$, $P_y\{\tau_{\delta} \leq T\}\to 0$ as $\delta \to 0$.
\item Let $\{Z(t)\}$ solve the  SDE 
\begin{equation*}
dZ(t)=\tilde \beta(Z(t))dt+\sigma (Z(t))dB(t)-a (Z(t))D\Gamma (Z(t))dt,
\end{equation*}%
where $\tilde \beta$ is as in \eqref{eq:tilbeta} and $(\Gamma ,\gamma )$ is a solution of \eqref{eq:hjbfirststar}.
For $\delta\in (0,\infty)$ define $\tau_{\delta}$ as in part (a) but with $Y$  replaced with $Z$.
Then for $z \in G$  and any $T\in (0,\infty )$, $P_z\{\tau_{\delta} \leq T\}\to 0$ as $\delta \to 0$.
\end{enumerate}
\end{lemma}

\begin{proof}
Consider first part (a). The argument is based on constructing an appropriate martingale. We use that
\begin{equation}\label{eq:bdrexpl}
\Gamma(x_{i})\uparrow\infty\text{ whenever }x_{i}\rightarrow\bar{x}\in\partial
G\text{ with }x_{i}\in G,
\end{equation}
that $\Gamma$ is $\mathcal{C}^{2}$ on $G$, and that by \eqref{eq:hjbfirst}, and since $a$ is nonnegative definite,
\[
\tilde{\mathcal{L}}\Gamma(x) \doteq b(x)\cdot D\Gamma(x)-a(x)D\Gamma(x)\cdot
D\Gamma(x)+\frac{1}{2}\text{tr[}a(x)D^{2}\Gamma(x)]\leq\gamma\text{ for }x\in G.
\]
Without loss we also assume $\Gamma\geq0$. Fix $T<\infty$ and
suppose that 
\begin{equation}\label{eq:contra}
\liminf_{\delta \to 0}P_{y}\{\tau_{\delta} \leq T\} \doteq \kappa >0.\end{equation}
Choose $M \in (0,\infty)$ such that $\Gamma(y) + \gamma T < \frac{1}{2}\kappa M$.
From \eqref{eq:bdrexpl} we can find $\delta_0 \in (0,\infty)$ so that $$\Gamma(x) \ge M \mbox{ whenever } x\in G, \; \dist(x,\partial G) \le \delta_0.$$
Then, for any $\delta >0$,
\begin{align*}
\Emb_{y}\Gamma(Y(\tau_{\delta}\wedge T))-\Gamma(y)  =\Emb_{y}\int_{0}^{\tau_{\delta
}\wedge T}\tilde{\mathcal{L}}\Gamma(Y(s))ds \leq \gamma T 
\end{align*}
and therefore for any $0< \delta \le\delta_0$
\begin{align*}
M \Pmb_y(\tau_{\delta} \le T)
&\le \Emb_y(\Gamma(Y(\tau_{\delta}))1_{\tau_{\delta}\le T})
\le \Emb_{y}\Gamma(Y(\tau_{\delta}\wedge T)) \le \Gamma(y) + \gamma T < \frac{1}{2}\kappa M.
\end{align*}
In particular, $\liminf_{\delta \to 0}P_{y}\{\tau_{\delta} \leq T\} \le \kappa/2$ which contradicts \eqref{eq:contra}.
The result follows. 

The proof of part (b) is similar on using the fact that $\sup_{x\in G} |\tilde c (x)| \doteq c_0 <\infty$ and that,
 if $(\Gamma ,\gamma )$ is a solution of \eqref{eq:hjbfirststar}, then
 \[
 \tilde{\mathcal{L}}\Gamma(x) \doteq \tilde \beta(x)\cdot D\Gamma(x)-a(x)D\Gamma(x)\cdot
 D\Gamma(x)+\frac{1}{2}\text{tr[}a(x)D^{2}\Gamma(x)]\leq (\gamma + c_0)\text{ for }x\in G.
 \]
We omit the details.

\end{proof}
\subsection{Proof of Theorem \protect\ref{thm:ergcontandexitrate}}

Recall that $\psi $ satisfies the equation 
\begin{equation}
\begin{aligned} {\mathcal{L}} \psi(x) &= - \lambda \psi(x), \;\; x \in G,\\
\psi(x) &= 0 , \;\; x \in \partial G. \end{aligned}  \label{eq:linpde1}
\end{equation}%
It is easy to verify that $\Psi =-\log \psi $ satisfies 
\begin{equation}\label{eq:xixj}
D\psi =-\psi D\Psi ,\;\;\psi _{x_{i}x_{j}}=-\psi (\Psi _{x_{i}x_{j}}-\Psi
_{x_{i}}\Psi _{x_{j}})
\end{equation}%
for all $1\leq i,j\leq d$. Thus 
\begin{equation*}
{\mathcal{L}}\psi (x)=b(x)\cdot D\psi (x)+\frac{1}{2}\mbox{tr}\lbrack
a(x)D^{2}\psi (x)]=-\psi \left( b(x)\cdot D\Psi(x) +\frac{1}{2}\mbox{tr}\lbrack
a(x)D^{2}\Psi (x)]-\frac{1}{2}D\Psi (x)\cdot a(x)D\Psi (x)\right) .
\end{equation*}%
Substituting this in the first line of \eqref{eq:linpde1} and recalling that 
$\psi (x)>0$ for $x\in G$, we have 
\begin{equation}
b(x)\cdot D\Psi (x)+\frac{1}{2}\mbox{tr}\lbrack a(x)D^{2}\Psi (x)]-\frac{1}{2%
}D\Psi (x)\cdot a(x)D\Psi (x)=\lambda ,\;\mbox{ for all }x\in G.
\label{eq:nonlin1}
\end{equation}%
Note that, for any $v\in {\mathbb{R}}^{d}$ 
\begin{equation*}
\min_{u\in {\mathbb{R}}^{d}}\left[ (b(x)+\sigma (x)u)\cdot v+\frac{1}{2}%
\left\Vert u\right\Vert ^{2}\right] =b(x)\cdot v-\frac{1}{2}v\cdot a(x)v.
\end{equation*}%
Using this in \eqref{eq:nonlin1} we have 
\begin{equation}
\frac{1}{2}\mbox{tr}\lbrack a(x)D^{2}\Psi (x)]+\min_{u\in {\mathbb{R}}^{d}}%
\left[ (b(x)+\sigma (x)u)\cdot D\Psi (x)+\frac{1}{2}\left\Vert u\right\Vert
^{2}\right] =\lambda ,\;\mbox{ for all }x\in G.  \label{eq:nonlin2}
\end{equation}%
Also, since $\psi (x)\rightarrow 0$ as $x\rightarrow \partial G$, we have
that $\Psi (x)\rightarrow \infty $ as $x\rightarrow \partial G$. Combining
these facts we have part 1 of the theorem.

Now consider part 2. Let $(\tilde{\Psi},\tilde{\lambda})\in {\mathcal{C}}%
^{2}(G)\times (0,\infty )$ be a solution of the HJB equation in %
\eqref{eq:hjbfirst}. Let $\tilde{\psi}(x)\doteq e^{-\tilde{\Psi}(x)}$ for $x\in G
$ and $\tilde{\psi}(x)\doteq 0$ for $x\in \partial G$. Then $\tilde{\psi}(x)\in {%
\mathcal{C}}^{2}(G)\cap {\mathcal{C}}(\bar{G})$ satisfies $\tilde{\psi}(x)>0$
for all $x\in G$; also it is easily verified that $(\tilde{\psi},\tilde{\lambda})
$ solve \eqref{eq:linpde1}. By the Krein-Rutman theorem once more (cf. \cite[%
Theorem 1.3]{du2006order}) we must have that for some $\kappa _{0}\in
(0,\infty )$, $\tilde{\psi}(x)=\kappa _{0}\psi (x)$ for all $x\in \bar{G}$ and 
$\lambda =\tilde{\lambda}$. This proves part 2 with $C=-\log \kappa _{0}$.

Consider now part 3. For $\delta \in (0,\infty )$, let 
\begin{equation*}
G_{\delta }=\{x\in G:\dist(x,\partial G)\geq \delta \}.
\end{equation*}%
Then, from our assumptions on the coefficients it follows that for every $%
y\in G_{\delta }$ there is a continuous $\{{\mathcal{F}}_{t}\}$%
-adapted process $Y^{\delta }(t)$ that is  unique in the strong sense \cite{KaratzasShreve1991brownian}
such that if $\tau _{\delta }\doteq \inf
\{t\geq 0:Y^{\delta }(t)\in \partial G_{\delta }\}$, then for all $0\leq
t<\tau _{\delta }$, 
\begin{equation*}
Y^{\delta }(t)=y+\int_{0}^{t}b(Y^{\delta }(s))ds+\int_{0}^{t}\sigma
(Y^{\delta }(s))dB(s)-\int_{0}^{t}a(Y^{\delta }(s))D\Psi (Y^{\delta }(s))ds.
\end{equation*}%
From the uniqueness it follows that for $0<\delta _{1}<\delta _{2}<\infty $, 
\begin{equation*}
Y^{\delta _{1}}(t)=Y^{\delta _{2}}(t)\mbox{ for all }0\leq t<\tau _{\delta
_{2}}.
\end{equation*}%
Also, note that $G=\cup _{\delta >0}G_{\delta }$. In order to complete the
proof of part 3 it suffices to show that, for every $0<t<\infty$,
\begin{equation*}
{P}_{y}(\tau _{\delta _{1}}<t)\rightarrow 0\mbox{ as }\delta
_{1}\rightarrow 0.
\end{equation*}%
However this is an immediate consequence of Lemma \ref{lem:lemstayin}. This proves 3.

In order to show that $\mu (dx)\doteq \varphi (x)\psi (x)dx$ is an invariant
measure for $Y$ defined in \eqref{contback2}, it suffices to show that 
\begin{equation}
({\mathcal{L}}_{Y}^{\ast }[\varphi \psi ])(x)=0\mbox{ for all }x\in G,
\end{equation}%
where ${\mathcal{L}}_{Y}^{\ast }$ is the formal adjoint of the infinitesimal
generator ${\mathcal{L}}_{Y}$ of the Markov process $Y$ in \eqref{contback2}. 
To simplify notation we suppress the independent variable $x$.
Then for $f\in {\mathcal{C}}_{0}^{2}(G)$
\begin{equation}
{\mathcal{L}}_{Y}f=b\cdot Df-Df\cdot aD\Psi +\frac{1}{2}%
\mbox{tr}[aD^{2}f],
\end{equation}%
and therefore for $h\in {\mathcal{C}}_{0}^{2}(G)$
\begin{equation}
{\mathcal{L}}_{Y}^{\ast }h=\hat{\beta}\cdot Dh+\frac{1}{2}\mbox{tr}%
\lbrack aD^{2}h]-\hat{c}h,
\end{equation}%
where with $\beta $ is as defined in \eqref{eq:defbeta}
\begin{equation}
 \hat{\beta}=\beta -\frac{1}{\psi }aD\psi ,
 \quad \hat{\beta}_{i}=-b_{i}-\frac{1}{\psi} (aD\psi)_{i}+\sum_{j=1}^{d}(a_{ij})_{x_{j}},
\label{eqn:forbetasnew}
\end{equation}
and with $c$  is defined in \eqref{eq:cterm}
\begin{equation}\label{eqn:chat}
\begin{aligned}
\hat{c}& =-\frac{1}{2}\sum_{i,j=1}^{d}(a_{ij})_{x_{i}x_{j}}+%
\sum_{i=1}^{d}(b_{i})_{x_{i}}-\sum_{i=1}^{d}\sum_{j=1}^{d}[a_{ij}\Psi
_{x_{j}}]_{x_{i}} 
=c-\sum_{i=1}^{d}\sum_{j=1}^{d}a_{ij}\Psi
_{x_{i}x_{j}}-\sum_{i=1}^{d}\sum_{j=1}^{d}(a_{ij})_{x_{i}}\Psi
_{x_{j}}.
\end{aligned}
\end{equation}
We will use\bigskip\ $\Psi _{x_{j}}=-\psi _{x_{j}}/\psi $, as well as the
identity
\begin{equation*}
\varphi \psi \sum_{i,j=1}^{d}a_{ij}\Psi _{x_{i}x_{j}}=-\varphi \mbox{tr}%
\lbrack aD^{2}\psi ]+\frac{\varphi }{\psi }D\psi \cdot aD\psi ,
\end{equation*}%

To show $\mathcal{L}_{Y}^{\ast }[\varphi \psi ]=0$, since $\mathcal{L}\psi
=-\lambda \psi $ and $\mathcal{L}^{\ast }\varphi =-\lambda \varphi $ is
suffices to establish 
\begin{equation}
\hat{\beta}\cdot D[\varphi \psi ]+\frac{1}{2}\mbox{tr}\lbrack aD^{2}\psi
\varphi ]-\hat{c}\varphi \psi \stackrel{?}{=}\psi \mathcal{L}^{\ast }\varphi
-\varphi \mathcal{L}\psi ,
\end{equation}%
where $?$ indicates that the equality is not yet established. Expanding
derivatives and using the definitions of $\mathcal{L}$ and $\mathcal{L}%
^{\ast }$ gives the equivalent statement 
\begin{equation*}
\begin{aligned}
&\psi \hat{\beta}\cdot D\varphi +\varphi \hat{\beta}\cdot D\psi +\frac{1}{2}%
\psi \mbox{tr}\lbrack aD^{2}\varphi ]+\frac{1}{2}\varphi \mbox{tr}\lbrack
aD^{2}\psi ]+D\varphi \cdot aD\psi -\hat{c}\varphi \psi  \\
&\qquad \stackrel{?}{=}\psi \beta \cdot D\varphi +\frac{1}{2}\psi \mbox{tr}\lbrack
aD^{2}\varphi ]-c\varphi \psi -\varphi b\cdot D\psi -\frac{1}{2}\varphi %
\mbox{tr}\lbrack aD^{2}\psi ].
\end{aligned}
\end{equation*}%
Cancellation of $\frac{1}{2}\psi \mbox{tr}\lbrack aD^{2}\varphi ]$ and then
substitution from \eqref{eqn:forbetasnew} gives the statement%
\begin{equation*}
\begin{aligned}
&\psi \beta \cdot D\varphi -aD\psi \cdot D\varphi -\varphi b\cdot D\psi -%
\frac{\varphi }{\psi }aD\psi \cdot D\psi +\varphi
\sum_{i,j=1}^{d}(a_{ij})_{x_{j}}\psi _{x_{i}}+\frac{1}{2}\varphi \mbox{tr}%
\lbrack aD^{2}\psi ]+D\varphi \cdot aD\psi -\hat{c}\varphi \psi  \\
&\qquad \stackrel{?}{=}\psi \beta \cdot D\varphi -c\varphi \psi -\varphi b\cdot D\psi -%
\frac{1}{2}\varphi \mbox{tr}\lbrack aD^{2}\psi ].
\end{aligned}
\end{equation*}%
Next we cancel where obvious, bring the $\frac{1}{2}\varphi \mbox{tr}\lbrack
aD^{2}\psi ]$ term to the right hand side, substitute for $\hat{c}$ and use
\eqref{eqn:chat} to get 
\begin{equation*}
-\frac{\varphi }{\psi }aD\psi \cdot D\psi +\varphi
\sum_{i,j=1}^{d}(a_{ij})_{x_{i}}\psi _{x_{j}}-c\varphi \psi -\varphi %
\mbox{tr}\lbrack aD^{2}\psi ]+\frac{\varphi }{\psi }D\psi \cdot aD\psi
-\varphi \sum_{i,j=1}^{d}(a_{ij})_{x_{j}}\psi _{x_{i}} 
\stackrel{?}{=}-c\varphi \psi -\varphi \mbox{tr}\lbrack aD^{2}\psi ].
\end{equation*}%
Cancellation now gives the desired result.
This completes the proof of part 4.

Now let $Y$ be as given in part 3 and for $\delta \in (0,\infty )$ let 
\begin{equation}
\tau _{\delta }\doteq \inf \{t>0:\mbox{dist}(Y(t),\partial G)\leq \delta \}.
\label{eq:taudeldef}
\end{equation}%
By It\^{o}'s formula, and since $\Psi $ satisfies \eqref{eq:nonlin1}, for
any $t>0$ 
\begin{equation}\label{eq:itoypsi}
\begin{aligned}
\Psi (Y(t))-\Psi (Y(0))& =\int_{0}^{t}b(Y(s))\cdot D\Psi
(Y(s))ds+\int_{0}^{t}D\Psi (Y(s))\cdot \sigma (Y(s))dB(s) \\
& \quad -\int_{0}^{t}D\Psi (Y(s))\cdot a(Y(s))D\Psi (Y(s))ds+\frac{1}{2}%
\int_{0}^{t}\mbox{tr}\lbrack D^{2}\Psi (Y(s))a(Y(s))]ds \\
& =\int_{0}^{t}D\Psi (Y(s))\cdot \sigma (Y(s))dB(s)-\frac{1}{2}\int_{0}^{t}%
\left[ D\Psi (Y(s))\cdot a(Y(s))D\Psi (Y(s))\right] ds+\lambda t.
\end{aligned}
\end{equation}
Applying the above identity with $t=T\wedge \tau _{\delta }$, taking
expectations, and dividing by $T$ throughout 
\begin{equation}
\begin{aligned} \frac{1}{2T} {E_y} \int_0^{T\wedge \tau_{\delta}} D
\Psi(Y(s))\cdot  a(Y(s)) D\Psi(Y(s)) ds &= \lambda \frac{\Emb_y(T\wedge
\tau_{\delta})}{T} + \frac{\Psi(y)- \Emb_y\Psi(Y(T\wedge \tau_{\delta}))}{T}\\
&\le \lambda + \frac{\Psi(y)+ \sup_{y \in G}(\Psi(y))_{-}}{T}. \end{aligned}
\label{eq:upplam}
\end{equation}%
Sending $\delta \rightarrow 0$, we have by the monotone convergence theorem
that 
\begin{equation*}
\frac{1}{2T}{E}_{y}\int_{0}^{T}\left[ D\Psi (Y(s))\cdot
a(Y(s))D\Psi (Y(s))\right] ds\leq \lambda +\frac{\Psi (y)+\sup_{y\in G}(\Psi
(y))_{-}}{T}.
\end{equation*}%
Recalling that the process $u$ defined in part 3 is in ${\mathcal{A}}(y)$,
we now have that 
\begin{equation}
J(y)\leq \limsup_{T\rightarrow \infty }\frac{1}{2T}{E}%
_{y}\int_{0}^{T}\Vert u(s)\Vert ^{2}ds=\limsup_{T\rightarrow \infty }\frac{1%
}{2T}{E}_{y}\int_{0}^{T}\left[ D\Psi (Y(s))\cdot a(Y(s))D\Psi (Y(s))%
\right] ds\leq \lambda .  \label{eq:uppbdonv}
\end{equation}

We now consider the reverse direction. Suppose that $u\in {\mathcal{A}}(y)$
and consider the corresponding state process $Y$ defined according to %
\eqref{contback}. 
Under Assumption \ref{assu:coeff},
it follows from \cite{linir} that in terms of the function $\mbox{dist}(x,G)$
we can define a sequence of bounded
domains $\{G_n\}_{n \in {\mathbb{N}}}$ such that $\dist(G_n^c, G) \to 0$ as $%
n\to \infty$, and for all $n \in {\mathbb{N}}$, $G_n \supset G_{n+1} \supset
G$ , $\dist(G_n^c, G) >0$, $G_n$ is ${\mathcal{C}}^{2,\alpha}$ for some $%
\alpha \in (0,1]$ that is independent of $n$.

For $n\in {\mathbb{N}}$, let $\lambda _{n}\in (0,\infty )$
be the principal  eigenvalue of $-{\mathcal{L}}$ associated with
the domain $G_{n}$ (so that any other  eigenvalue $\tilde{\lambda}%
_{n}$ of $-{\mathcal{L}}$ associated with the domain $G_{n}$ satisfies $%
\lambda _{n}\leq \mbox{Re}(\tilde{\lambda}_{n})$). It is known that $\lambda
_{n}$ is a decreasing sequence (cf. \cite[Proposition 2.3]{berros}) and as $%
n\rightarrow \infty $ $\lambda _{n}\rightarrow \lambda $. The latter
property follows, for example, from \cite[Theorem 1.10]{berros}.

Let $\psi _{n}$ be a strictly positive eigenvector associated with the
eigenvalue $\lambda _{n}$ and let $\Psi _{n}\doteq -\log \psi _{n}$. Note
that $\Psi _{n}\in {\mathcal{C}}^{2}(G_{n})$ and so in particular its
restriction to $\bar{G}$ is in ${\mathcal{C}}^{2}(\bar{G})$. Applying It\^{o}%
's formula to $\Psi _{n}(Y(t))$, and noting that \eqref{eq:nonlin2} is
satisfied with $(\Psi ,\lambda ,G)$ replaced with $(\Psi _{n},\lambda
_{n},G_{n})$, we have, for any $T>0$, 
\begin{align*}
\Psi _{n}(Y(T))-\Psi _{n}(Y(0))& =\int_{0}^{T}b(Y(s))\cdot D\Psi
_{n}(Y(s))ds+\int_{0}^{T}D\Psi _{n}(Y(s))\cdot \sigma (Y(s))dB(s) \\
& \quad +\int_{0}^{T}D\Psi _{n}(Y(s))\cdot \sigma (Y(s))u(s)ds+\frac{1}{2}%
\int_{0}^{T}\mbox{tr}\lbrack D^{2}\Psi _{n}(Y(s))a(Y(s))]ds \\
& \geq \int_{0}^{T}D\Psi _{n}(Y(s))\cdot \sigma (Y(s))dB(s)-\frac{1}{2}%
\int_{0}^{T}\Vert u(s)\Vert ^{2}ds+\lambda _{n}T.
\end{align*}%
Thus 
taking expectations, and dividing by $T$ throughout, we have 
\begin{equation}\label{eq:otherdir}
\frac{1}{2T}{E}_{y}\int_{0}^{T}\Vert u(s)\Vert ^{2}ds\geq \lambda
_{n}+\frac{1}{T}(\Psi _{n}(y)-{E_y}\Psi _{n}(Y(T))).
\end{equation}%
Finally, sending $T\to \infty$ and recalling that $\Psi_n$
is bounded on $\bar G$ and $Y(T) \in G$ a.s., 
\begin{equation*}
J(y)=\inf_{u\in {\mathcal{A}}(y)}\limsup_{T\rightarrow \infty }\frac{1}{2T}{%
E}_{y}\int_{0}^{T}\Vert u(s)\Vert ^{2}ds\geq \lambda _{n}.
\end{equation*}%
Sending $n\rightarrow \infty $ we get that $J(y)\geq \lambda $ which
completes the proof of part 5.

Part 6 is now immediate from part 5 and the inequality in \eqref{eq:uppbdonv} and \eqref{eq:otherdir}.

Finally, the proof of \eqref{eq:lamdainv} can be completed as in %
\eqref{eq:upplam} and in \eqref{eq:otherdir} by taking $Y$ to be as given by
the equation in \eqref{contback2} but with $Y(0)$ distributed according to
the stationary distribution $\mu$ (and observing that
$\int_G |\Psi(y)| \mu(dy) <\infty$).

\hfill \qed
\section{Proof of Theorem \ref{thm:ergcontandqsd}}
Since many arguments are similar to that in the proof of Theorem \ref{thm:ergcontandexitrate}, we will omit some details. Recall 
from Theorem \ref{thm:llstar} that $(\tilde \varphi, \lambda)$ satisfy
\begin{equation}  \label{eq:evlstarz}
\begin{aligned} -{\mathcal{L}^*_{\dens}} \phi(x) &= \lambda \phi(x), & x \in G\\
\phi(x) &= 0, & x \in \partial G \end{aligned}
\end{equation}
Note that, with $\tilde \Phi = - \log \tilde \varphi$, \eqref{eq:xixj} holds with $(\psi, \Psi)$ replaced with
$(\tilde \varphi, \tilde \Phi)$. Thus, recalling \eqref{eq:generadjden}, 
\begin{equation*}
	\begin{aligned}
{\mathcal{L}^*_{\dens}}\tilde\varphi (x)&= \tilde \beta(x)\cdot D\tilde\varphi (x)+\frac{1}{2}\mbox{tr}\lbrack
a(x)D^{2}\tilde\varphi (x)] - \tilde c(x) \tilde\varphi (x)\\
&=-\tilde \varphi(x) \left( \tilde \beta(x)\cdot D\tilde\Phi +\frac{1}{2}\mbox{tr}\lbrack
a(x)D^{2}\tilde \Phi (x)]-\frac{1}{2}D\tilde\Phi (x)\cdot a(x)D\tilde\Phi (x) +\tilde c(x)\right) .
\end{aligned}
\end{equation*}%
Since $\tilde \varphi(x)>0$ for all $x \in G$, we have from \eqref{eq:evlstarz} that
$$\tilde \beta(x)\cdot D\tilde\Phi +\frac{1}{2}\mbox{tr}\lbrack
a(x)D^{2}\tilde \Phi (x)]-\frac{1}{2}D\tilde\Phi (x)\cdot a(x)D\tilde\Phi (x) +\tilde c(x) = \lambda$$
for all $x\in G$ from which part 1 of the theorem follows as in the proof of Theorem \ref{thm:ergcontandexitrate}.

Parts 2 and 3 follow exactly as in Theorem \ref{thm:ergcontandexitrate} by using Krein-Rutman theorem  (cf. \cite[Theorem 1.3]{du2006order}) and Lemma \ref{lem:lemstayin} (b).

Now we consider part 4.  It suffices to show that 
\begin{equation}
({\mathcal{L}}_{Z}^{\ast }[\varphi \psi ])(x)=0\mbox{ for all }x\in G,
\end{equation}%
where ${\mathcal{L}}_{Z}^{\ast }$ is the formal adjoint of the infinitesimal
generator ${\mathcal{L}}_{Z}$ of the Markov process $Z$ in \eqref{contbackstar2}. 
As in the last section, we simplify notation by supressing the independent variable. 
For $f\in {\mathcal{C}}_{0}^{2}(G)$
\begin{equation}
{\mathcal{L}}_{Z}f=\tilde \beta\cdot Df-Df\cdot aD\tilde\Phi +\frac{1}{2}%
\mbox{tr}\lbrack aD^{2}f],
\end{equation}%
and therefore for $h\in {\mathcal{C}}_{0}^{2}(G)$
\begin{equation}
{\mathcal{L}}_{Z}^{\ast }h=\tilde{\beta}^*\cdot Dh+\frac{1}{2}\mbox{tr}%
\lbrack aD^{2}h]-{c}^*h,
\end{equation}%
where using   \eqref{eq:defbeta} and \eqref{eq:tilbeta}
\begin{equation*}
\tilde{\beta}^{\ast }=b-\frac{1}{\varsigma }aD\varsigma +aD\tilde{\Phi}
\end{equation*}%
and 
\begin{equation*}
c^{\ast }=-c+\frac{1}{\varphi }\mbox{tr}\lbrack aD^{2}\varphi ]-\frac{1}{%
\varphi ^{2}}D\varphi \cdot aD\varphi +\sum_{ij=1}^{d}(a_{ij})_{x_{i}}\frac{%
\varphi _{x_{j}}}{\varphi }.
\end{equation*}%
Using $\tilde \varphi  = \varphi /\dens$ (see Theorem \ref{thm:llstar})
\begin{equation*}
D\tilde{\Phi}=D[-\log \varphi /\dens]=\frac{1}{\dens }D\dens -%
\frac{1}{\varphi }D\varphi ,
\end{equation*}%
it follows that 
\begin{equation*}
\tilde{\beta}^{\ast }=b-\frac{1}{\varsigma }aD\varsigma +\frac{1}{\varsigma }%
aD\varsigma -\frac{1}{\varphi }aD\varphi =b-\frac{1}{\varphi }aD\varphi .
\end{equation*}

To show $\mathcal{L}_{Z}^{\ast }[\varphi \psi ]=0$, since $\mathcal{L}\psi
=-\lambda \psi $ and $\mathcal{L}^{\ast }\varphi =-\lambda \varphi $ it
suffices to show 
\begin{equation*}
\tilde{\beta}^{\ast }\cdot D[\varphi \psi ]+\frac{1}{2}\mbox{tr}\lbrack
aD^{2}[\varphi \psi ]]-c^{\ast }\varphi \psi \stackrel{?}{=} \varphi \mathcal{%
L}\psi -\psi {\mathcal{L}}^{\ast }\varphi ,
\end{equation*}%
where again the $?$ indicates the equality has yet to be established.
Substituting for the generators and expanding gives the equivalent statement%
\begin{equation*}
\begin{aligned}
&\psi b\cdot D\varphi +\varphi b\cdot D\psi -\frac{\psi }{\varphi }%
aD\varphi \cdot D\varphi -aD\varphi \cdot D\psi +\frac{1}{2}\psi \mbox{tr}%
\lbrack aD^{2}\varphi ]+\frac{1}{2}\varphi \mbox{tr}\lbrack aD^{2}\psi ] \\
&\qquad +D\varphi \cdot aD\psi +c\varphi \psi -\psi \mbox{tr}\lbrack D^{2}\varphi
a]+\frac{\psi }{\varphi }D\varphi \cdot aD\varphi -\psi
\sum_{ij=1}^{d}(a_{ij})_{x_{i}}\varphi _{x_{j}} \\
&\quad \stackrel{?}{=}\varphi b\cdot D\psi +\frac{1}{2}\varphi \mbox{tr}\lbrack
aD^{2}\psi ]+\psi b\cdot D\varphi -\psi
\sum_{ij=1}^{d}(a_{ij})_{x_{j}}\varphi _{x_{i}}-\frac{1}{2}\psi \mbox{tr}%
\lbrack aD^{2}\varphi ]+c\varphi \psi .
\end{aligned}
\end{equation*}%
Cancellations then give the statement 
\begin{equation*}
-aD\varphi \cdot D\psi +D\varphi \cdot aD\psi\stackrel{?}{=}0,
\end{equation*}%
and the result follows from the symmetry of $a$.
This completes the proof of 4.

For proof of 5 we first argue $J^*(z)\le \lambda $ and then $J^*(z)\ge \lambda$. For the first inequality we introduce
$\tau_{\delta}$ as in \eqref{eq:taudeldef} with $Y$ replaced with $Z$ and apply It\^{o}'s formula to
$\tilde \Phi(Z(t))$ in a manner similar to \eqref{eq:itoypsi}. Then as in \eqref{eq:upplam} we see that
\begin{equation}
\begin{aligned} \frac{1}{2T} {E_y} \int_0^{T\wedge \tau_{\delta}} \left(D
\tilde\Phi(Z(s))\cdot  a(Z(s)) D\tilde\Phi(Y(s)) + \tilde c(Z(s))\right) ds 
&\le \lambda + \frac{\tilde\Phi(y)+ \sup_{y \in G}(\tilde\Phi(y))_{-}}{T}. \end{aligned}
\label{eq:upplam2}
\end{equation}%
The inequality $J^*(z)\le \lambda $ follows as in the proof of Theorem \ref{thm:ergcontandexitrate}. 
For the reverse inequality we once more consider domains $G_n \supset G_{n+1} \supset
G$ as in Theorem \ref{thm:ergcontandexitrate} and principle eigenvalue-eigenvector pair $(\lambda_n, \tilde \varphi_n)$
associated with $-\mathcal{L}^*_{\dens}$. We then apply It\^{o}'s formula to
$\tilde \Phi_n(Z(t))$ where $\tilde \Phi_n = -\log \tilde \varphi_n$ and $Z$ is given by \eqref{contbackstarden}
with an arbitrary $u \in {\mathcal{A}}^{* }(z)$. Then as in the proof of \eqref{eq:otherdir} we get
\begin{equation}\label{eq:otherdirb}
\frac{1}{2T}{E}_{z}\int_{0}^{T}\left(\Vert u(s)\Vert ^{2} + \tilde c(Z(s)) \right) ds\geq \lambda
_{n}+\frac{1}{T}(\tilde\Phi _{n}(y)-{E_y}\tilde\Phi _{n}(Y(T))).
\end{equation}%
The proof of the inequality $J^*(z)\ge \lambda$ is now concluded as in the proof of Theorem \ref{thm:ergcontandexitrate}.
This completes the proof of part 5. Part 6 is now immediate from part 5.
\hfill \qed
\vspace{\baselineskip}

\noindent \textbf{Acknowledgement:} 
The research of AB was supported in part by the NSF (DMS-1814894 and DMS-1853968).
The research of PD was supported in part by the National Science Foundation (DMS-1904992) and the AFOSR (FA-9550-18-1-0214). 
The research of PN was supported in part by the Swedish Research Council (VR-2018-07050) and the Wallenberg AI, Autonomous Systems and Software Program (WASP) funded by the Knut and Alice Wallenberg Foundation. The research of GW was supported in part by the Swedish e-Science Research Centre through the Data Science MCP.



\vspace{\baselineskip}

\noindent{\scriptsize {\textsc{\noindent A. Budhiraja\newline
Department of Statistics and Operations Research\newline
University of North Carolina\newline
Chapel Hill, NC 27599, USA\newline
email: budhiraj@email.unc.edu \vspace{\baselineskip} } }}

{\scriptsize \noindent\textsc{\noindent P. Dupuis \newline
Division of Applied Mathematics\newline
Brown University\newline
Providence, RI 02912, USA\newline
email: paul\_dupuis@brown.edu \vspace{\baselineskip} } }

{\scriptsize \noindent\textsc{\noindent P. Nyquist and G.-J. Wu%
\newline
Department of Mathematics\newline
KTH Royal Institute of Technology\newline
100 44 Stockholm, Sweden\newline
email: pierren@kth.se, gjwu@kth.se \vspace{\baselineskip}} }

\end{document}